\documentclass[10pt, letterpaper]{amsart}

\usepackage{amsmath}
\usepackage{amssymb}
\usepackage{amsthm}
\usepackage[all]{xypic}
\usepackage{mathrsfs}
\usepackage{enumerate}
\usepackage{graphicx}
\usepackage{url}
\usepackage[mathscr]{eucal}
\usepackage{tikz}
\usepackage{microtype}

\DeclareFontFamily{OMS}{rsfs}{\skewchar\font'60}
\DeclareFontShape{OMS}{rsfs}{m}{n}{<-5>rsfs5 <5-7>rsfs7 <7->rsfs10 }{}
\DeclareSymbolFont{rsfs}{OMS}{rsfs}{m}{n}
\DeclareSymbolFontAlphabet{\scr}{rsfs}

\newcommand{\sF}{\mathscr{F}}
\newcommand{\sG}{\mathscr{G}}
\newcommand{\sH}{\mathscr{H}}

\newcommand{\sO}{\mathscr{O}}
\newcommand{\sP}{\mathscr{P}}

\newcommand{\bP}{\mathbf{P}}

\newcommand{\fp}{\mathfrak{p}}
\newcommand{\fm}{\mathfrak{m}}

\newcommand{\fq}{\mathfrak{q}}

\DeclareMathOperator{\Ex}{{Ex}} 
\DeclareMathOperator{\snc}{{snc}} 
\newcommand{\dR}{\mathscr{R}} 

\DeclareMathOperator{\codim}{codim}
\DeclareMathOperator{\coker}{{coker}}

\DeclareMathOperator{\Spec}{{Spec}}

\theoremstyle{theorem} 
\newtheorem{theorem}{Theorem}[section]
\newtheorem*{theorem*}{Theorem}
\newtheorem{corollary}[theorem]{Corollary} 
\newtheorem{lemma}[theorem]{Lemma}
\newtheorem{proposition}[theorem]{Proposition}
\newtheorem*{proposition*}{Proposition}
\theoremstyle{definition} 
\newtheorem{definition}[theorem]{Definition}
\newtheorem{example}[theorem]{Example}
\newtheorem{notation}[theorem]{Notation}

\begin{document}

\begin{abstract}
  Rational pairs, recently introduced by Koll\'ar and Kov\'acs,
  generalize rational singularities to pairs $(X,D)$.
  Here $X$ is a normal variety and $D$ is a reduced divisor on $X$.
  Integral to the definition of a rational pair is the notion of a thrifty
  resolution, also defined by Koll\'ar and Kov\'acs,
  and in order to work with rational pairs it is often
  necessary to be able to tell whether a given resolution is thrifty.
  In this paper we present several foundational results
  that are helpful for identifying thrifty resolutions
  and analyzing their behavior.
  We also show that general hyperplane sections of
  rational pairs are again rational.
   
  In 1978, Elkik proved that rational singularities are deformation invariant.
  Our main result is an analogue of this theorem for rational pairs:
  given a flat family $X\to S$ and a Cartier divisor $D$ on $X$,
  if the fibers over a smooth point $s\in S$ form a rational
  pair, then $(X,D)$ is also rational near the fiber $X_s$.
\end{abstract}

\title[]{Deformation invariance of rational pairs}
\author{Lindsay Erickson}


\date{November 9, 2014}
\maketitle

\ 

\section{Introduction}

We present a proof that rational pairs are deformation invariant.
The notion of a \emph{rational pair},
recently introduced in \cite[Section 2.5]{KK13},
is a generalization of rational singularities from a variety $X$
to a reduced pair $(X,D)$.
A \emph{reduced pair} is a normal variety $X$
together with a Weil divisor $D$, all of whose
coefficients are 1.
Such a $D$ is a \emph{reduced divisor}.
The main results in this paper include the assumption that
the underlying field $k$ has characteristic 0.

The analogue of a smooth variety in the pairs setting is an \emph{snc pair}.
A pair $(X,D)$ has \emph{simple normal crossings}, or is snc, if
$X$ is smooth, every $D_i$ is smooth, and all the intersections of the components
$D_i$ are transverse.
If $(X,D)$ is not snc, we can still refer to the \emph{snc locus} $\snc(X,D)$,
which is the largest open set $U\subset X$ so that $(U,D\cap U)$ is snc.

\begin{definition}
  \label{def-birattrans}
Let $f\colon Y\to X$ be birational, where
$X$ is normal, and let $D\subset X$ be a divisor.
The birational transform of $D$, denoted
$f_\ast^{-1}D$, is a divisor in $Y$, and it is defined as follows:
let $U\subset X$ be the largest open set over which
$f$ is an isomorphism and an inverse $f^{-1}$ exists,
then map $D\cap U$ into $Y$ by the morphism $f^{-1}$,
and finally take the closure in $Y$ of the image.
\end{definition}

The birational transform $f_\ast^{-1}D$ is used
to define a resolution of $(X,D)$.

\begin{definition}
  \label{def-resofpairs}
A resolution of a pair $(X,D)$ is a 
resolution of singularities $f\colon Y\to X$,
such that the pair $(Y, f_\ast^{-1}D)$ has simple normal crossings.
\end{definition}

To keep the notation concise, we will write $B$ for the divisor $f_\ast^{-1}D$ in $Y$.
Note that $B$ never has any exceptional components: its components are in one-to-one
birational correspondence to the components of $D$.
We will use ``resolution'' to mean ``resolution of pairs'',
and we will refer to a resolution of $(X,D)$ with the notation $f\colon (Y,B)\to (X,D)$.
When we really mean a resolution of a variety $X$, and not of a pair $(X,D)$,
we'll specifically say as much.

If $(X,D=\bigcup D_i)$ is an snc pair, then
an irreducible component of any intersection of the $D_i$ is called a \emph{stratum}.
If $(X,D)$ is not snc, then we may consider the strata of its snc locus.
In order to determine whether a pair is rational, we'll need to examine
the strata of its snc locus and how they behave under a resolution of pairs.
A general reference for these definitions is
\cite{KK13}.
  
This new theory of rational pairs requires us to restrict our attention to
certain resolutions, called \emph{thrifty resolutions},
which were also introduced and developed in \cite{KK13}.
These resolutions of pairs have the nice properties
that are required to generalize
theorems about rational singularities of varieties to pairs.
Specifically, when we restrict to resolutions of pairs that are also thrifty,
we have an analogue of Grauert-Riemenschneider vanishing in characteristic 0
(see (\ref{GR-thrifty}) below),
and we know that if a pair has one thrifty rational resolution,
then every other thrifty resolution is also rational.
  
A thrifty resolution $f\colon (Y,B)\to (X,D)$,
which is defined in (\ref{def-thrifty}) and 
which we will discuss in much more detail below,
is a resolution in the sense of (\ref{def-resofpairs}) above
that satisfies two conditions:
\begin{enumerate}
\item (Condition 1) $f$ is an isomorphism over the generic point of any stratum of $\snc(X,D)$
\item (Condition 2) $f$ is an isomorphism at the generic point of any stratum of $(Y,B)$
\end{enumerate}

The first few results in this paper
simplify the task of verifying that a
given resolution is thrifty.
The first of these foundational results
is a simple criterion for Condition 2 in the definition of thriftiness,
which is usually the more difficult of the two conditions to check.
First, a definition.

\begin{definition}\label{def-logres}
  A resolution $f\colon Y\to X$, in the sense of (\ref{def-resofpairs}),
  is a weak log resolution if the exceptional locus $\Ex(f)$ is a divisor and
  the pair $(Y,B+\Ex(f))$ is snc.
  
  We will just call this a \emph{log resolution} throughout the paper,
  since we'll never have a reason to
  use any other type of log resolution in what follows.
\end{definition}

\begin{proposition}[see (\ref{condition2})]
  If $f\colon (Y,B)\to(X,D)$ is a log resolution as in (\ref{def-logres}),
  then $f$ satisfies Condition 2 for thriftiness.
\end{proposition}

We will also show that if a birational morphism $f\colon Y\to X$ is
an isomorphism over every stratum of $\snc(X,D)$,
then it
can be dominated by a thrifty log resolution of $(X,D)$.

\begin{theorem}[see (\ref{logres})]
  Suppose $f\colon Y\to X$ is a proper birational morphism
  between normal varieties, and
  $D\subset X$ is a reduced divisor.
  If $f$ is an isomorphism
  over every stratum of $\snc(X,D)$,
  then there is a thrifty log resolution of $(X,D)$ factoring through $f$.
\end{theorem}

It is not known whether a rational resolution (defined below in
(\ref{ratresofpairs})) is necessarily thrifty,
but we give a significant partial result in that direction:
it is true for log resolutions as in (\ref{def-logres}).
Indeed, every resolution can be dominated by a log resolution,
so our result suffices for many applications.

\begin{proposition}[see (\ref{rational-log-thrifty})]
  If a log resolution as in (\ref{def-logres}) of a pair is rational, then it is thrifty.
\end{proposition}

We also show that rational pairs, like rational singularities, satisfy a Bertini-type
theorem:
cutting a rational pair by a general hyperplane yields another rational pair.

\begin{theorem}[see (\ref{bertini-type})]
  If $(X,D)$ is a projective rational pair,
  then a general hyperplane section of $(X,D)$ is also rational.
\end{theorem}

The main result in this paper is on deformation invariance of rational pairs
in a flat family.
In \cite{Elkik}, Elkik proved that rational singularities are deformation invariant:
given a variety $X$ with rational singularities and a flat morphism $X\to S$,
then if $s$ is a smooth point and the fiber $X_s$ over $s$ has rational
singularities,
so does $X$ in a neighborhood of $X_s$.

In this paper we show that rational pairs $(X,D)$,
with $D$ a Cartier divisor in $X$,
are also deformation invariant.

\begin{theorem}[see (\ref{main-thm})]
  Let $(X, D)$ be a pair, with $D$ Cartier.
  Suppose $X\to S$ is a flat morphism, and $s\in S$ is a smooth point
  so that the fibers $(X_s, D_s)$ form a reduced pair.
  If the fiber $(X_s, D_s)$ is a rational pair,
  then the total space $(X,D)$ is a rational pair
  in a neighborhood of $(X_s, D_s)$.
  
  That is, if the fiber $(X_s,D_s)$ is rational at $x$, then
  the total space $(X,D)$ is also rational at $x$.
\end{theorem}

With some additional assumptions, we have a corollary
about fibers near $X_s$.

\begin{corollary}[see (\ref{main-corollary})]
  If, moreover, the morphism $X\to S$ is proper
  and $S$ is a curve,
  then there is a neighborhood $W\subset S$ of $s$
  such that for any $w\in W$,
  the pair of fibers $(X_w,D_w)$ is rational.
\end{corollary}

\ 

\section{Rational resolutions of pairs}

Recall that a resolution $f\colon Y\to X$ of a variety is rational
if the natural morphism $\sO_X\to\dR f_\ast\sO_Y$ is a quasi-isomorphism
and the higher direct images $R^if_\ast\omega_Y$ vanish for $i>0$.
(This second part holds automatically in characteristic 0
by the Grauert-Riemenschneider vanishing theorem (\cite{GR}).)

The definition of a rational resolution of a pair $(X,D)$ is
formally very similar to that of a rational resolution
of a variety $X$.
Again, throughout this paper
we assume all varieties are defined over a field of characteristic 0.
The definitions from \cite{KK13} also make sense in positive characteristic,
but the proofs of the main results do rely on the characteristic-0 assumption.

\begin{definition}[\protect{\cite[2.78]{KK13}}]
  Let $(X,D)$ be a reduced pair.
  A resolution $f\colon (Y, B)\to(X,D)$ is rational if
  \begin{enumerate}
  \item The natural map $\sO_X(-D)\to \dR f_\ast\sO_Y(-B)$
    is a quasi-isomorphism, and
  \item The higher direct images $R^if_\ast\omega_Y(B)$ vanish for $i>0$.
  \end{enumerate}
  \label{ratresofpairs}
\end{definition}

If $D$ is a Cartier divisor in $X$, then there is a characterization
of rational resolutions $f\colon (Y,B)\to (X,D)$
that looks very much like Kempf's well-known criterion
for rational resolutions of varieties $f\colon Y\to X$
(see \cite[2.77]{KK13}).

\begin{theorem}[\protect{\cite[2.84]{KK13}}]
  \label{kempf2}
  Let $(X, D)$ be a reduced pair, with $D$ Cartier, and
  let $f\colon (Y,B)\to (X,D)$ be a resolution.
  Then $f$ is rational if and only if two conditions are satisfied:
  \begin{enumerate}
  \item $X$ is CM, and
  \item The natural map $\dR f_\ast\omega_Y(B)\to\omega_X(D)$ is a quasi-isomorphism.
  \end{enumerate}
\end{theorem}

It is well known that if one resolution of a variety $X$ is rational, then
every other resolution of $X$ is also rational.
In this case we say that $X$ has \emph{rational singularities}.
The analogous statement does not hold for
rational resolutions of pairs as defined in (\ref{ratresofpairs}):
there are many pairs $(X,D)$ that have both rational
and non-rational resolutions.
For example, an snc pair $(X,D)$ that has a stratum of at least codimension 2
will have both types of resolutions; see \cite[p.\ 94]{KK13}.

The resolutions that cause this statement to fail
are all of a certain type.
If we exclude them,
then the analogous claim for pairs \emph{is} true.
So we will restrict our attention
to a certain type of resolution.
The resolutions that we'll consider are called \emph{thrifty}.

\begin{definition}[\protect{\cite[2.79]{KK13}}]
  \label{def-thrifty}
  Let $(X,D)$ be a reduced pair, with $f\colon (Y, B)\to (X, D)$ a resolution.
  Then $f$ is a thrifty resolution if two conditions hold:
  \begin{enumerate}
  \item $f$ is an isomorphism over the generic point of every stratum
    of $\snc(X,D)$; equivalently, $f(\Ex(f))$ does not contain any stratum of $\snc(X, D)$
  \item $f$ is an isomorphism at the generic point of every stratum
    of $(Y,B)$; equivalently, the exceptional locus $\Ex(f)$ does
    not contain any stratum of $(Y, B)$.
  \end{enumerate}
  We'll refer to these as Conditions 1 and 2 from now on.
\end{definition}

It is shown in \cite[2.86]{KK13} that in characteristic 0,
if a pair $(X,D)$ has a thrifty rational
resolution,
then every other thrifty resolution of $(X,D)$ is also rational.
As long as we restrict our attention to thrifty resolutions, then,
the situation for pairs is similar to the situation for varieties.

Now we have enough to define rational pairs.

\begin{definition}[\protect{\cite[2.80]{KK13}}]
  \label{rational-res-pairs} If $X$ is a normal variety, with $D$ a reduced
  divisor on $X$,
  then the pair $(X,D)$ is rational if it has a thrifty rational resolution.
  (Equivalently, if every thrifty resolution of $(X,D)$ is rational.)
\end{definition}

It is an open question whether a rational resolution is necessarily thrifty.
For dlt pairs this is true; see \cite[2.87]{KK13}.
In (\ref{rational-log-thrifty}) below, we give another partial answer:
given a resolution of pairs $f\colon  (Y,B)\to (X,D)$, if the entire preimage
of $D$ is an snc divisor in $Y$---that is, if $f$ is a log resolution
as in (\ref{def-logres})---then $f$ is indeed thrifty.

\ 

\section{Preliminary results: Thrifty resolutions}

In order to work with rational pairs, 
it will be essential
to be able to tell whether a given resolution is thrifty---that is,
whether it satisfies Conditions 1 and 2 from (\ref{def-thrifty}).
Condition 1 of thriftiness is a property of $\snc(X,D)$:
to see whether it holds, it is only necessary to check
the $\snc$ locus, where $X$, the $D_i$, and all the intersections of the $D_i$
are smooth.
Condition 2, on the other hand, is \emph{not} a property of $\snc(X,D)$,
so we must examine points that map to the non-snc locus of $(X,D)$.
Points outside of $\snc(X,D)$ are trickier to deal with, 
so alternative ways to
check for Condition 2 would be welcome.

Condition 2 is automatic if $f$ is a log resolution,
as we'll show below.
Such an $f$ is then thrifty if and only if
it satisfies Condition 1.
Recall that a \emph{log resolution} $(Y,B)\to(X,D)$,
as defined in (\ref{def-logres}), is
a resolution with two additional conditions:
the exceptional locus $\Ex(f)$ is a divisor in $Y$,
and the pair $(Y,B+\Ex(f))$ is snc.
This is a much stronger requirement than
that $(Y,B)$ be snc,
as in the definition of a resolution of pairs:
for $f$ to be a log resolution, 
the components of $\Ex(f)$ must be smooth divisors
and must intersect each other and
the components of $B$ transversally.

\begin{proposition}
  \label{condition2}
  If $f\colon (Y,B)\to(X,D)$ is a log resolution
  as in (\ref{def-logres}), then
  $f$ satisfies Condition 2.
\end{proposition}

\begin{proof}
  Suppose $f\colon (Y,B)\to(X,D)$ is a log resolution
  as defined in (\ref{def-logres}).
  Let $E$ be the reduced divisor supported on $\Ex(f)$, so that the pair
  $(Y,B+E)$ is snc and $B+E$ is reduced.
  Condition 2 fails exactly when a stratum of $(Y,B)$ is contained in $E$.
  Let $Z$ be a stratum of $(Y,B)$, so that $Z$ is a component
  of some intersection $\bigcap_{i\in I}D_i'$, where $I=i_1,\ldots, i_s$.
  Now we appeal to \cite[4.16.2]{KK13}, originally stated and proved in \cite[3.9.2]{Fujino},
  which says that an intersection of $s$ components of the divisor
  in a reduced dlt pair has pure codimension $s$.
  This theorem applies here because our $(Y,B)$ 
  is snc and therefore dlt,
  so $Z$ has codimension $s$.
  
  Now if $Z\subset E$, then $Z$ is contained in some component $E_j$ of $E$.
  But then $Z\subset E_j\cap\left(\bigcap_{i\in I}D_i'\right)$.
  Now $(Y,B+E)$ is also snc and hence dlt, so this intersection
  has codimension $s+1$ by another application of \cite[4.16.2]{KK13}.
  But then $Z$, which has codimension $s$, cannot be a subset of this intersection,
  so this situation is impossible.
  Thus Condition 2 always holds for a log resolution.
\end{proof}

The next result will be very useful in the proof of the main theorem (\ref{main-thm}).
In that argument, we analyze a pair $(X,D)$ and a certain subpair $(X_t,D_t)$, where $X_t$
is a Cartier divisor in $X$, and $D_t$ in $D$.
We'll appeal to the following theorem to show that the restriction of a thrifty
resolution of $(X,D)$ to $X_t$ is dominated by a thrifty resolution of $(X_t,D_t)$.

\begin{theorem}
  \label{logres}
  Suppose $f\colon Y\to X$ is a proper birational morphism
  between normal varieties, and
  $D\subset X$ is a reduced divisor.
  If $f$ is an isomorphism
  over every stratum of $\snc(X,D)$,
  then there is a thrifty log resolution of $(X,D)$ factoring through $f$.
\end{theorem}

\begin{proof}
  Let $B=f_\ast^{-1}D$ be the birational transform of $D$ in $Y$.
  We'll construct a thrifty log resolution $g$ of $(X,D)$ 
  starting with the pair $(Y,B)$.
  
  First, blow up the exceptional components of codimension at least 2
  in $Y$ of the birational
  morphism $f\colon Y\to X$.
  From this we obtain a new proper birational morphism $f'\colon Y'\to X$,
  and a new birational transform $B'$ of $D$ in $Y'$.
  Now the exceptional locus of $Y'\to X$ is a divisor in $Y'$,
  and we'll call it $E'$.
  
  Let $h$ be a log resolution of the pair $(Y',B'+E')$ that is
  an isomorphism over $\snc(Y',B'+E')$.
  To do this, we appeal to Szab\'o's theorem from \cite{Szabo},
  which says that choosing such a resolution is possible;
  see \cite[10.45.2]{KK13}.
  
  Now write $g=f\circ f'\circ h$.
  It is a log resolution of $X$ and the birational transform of $D$
  by $g$ is an snc divisor.
  We'll show now that $g$ is also an isomorphism
  over the generic point of every stratum of $\snc(X,D)$.
  
  Let $U\subset X$ be the open set over which $f$ is an isomorphism.
  Let $x$ be the generic point of any stratum of $\snc(X,D)$.
  Then by assumption $x\in U$, so $x\in U\cap\snc(X,D)$.
  The pair $(Y,B)$ is snc at the preimage of $x$ in $Y$,
  because $f$ is an isomorphim there.
  Similarly, $(Y',B')$ is snc at the preimage of $x$ in $Y'$:
  again, the blowups $f'\colon Y'\to Y$ do not affect points outside
  the exceptional locus of $f$,
  so the entire preimage of $U\cap\snc(X,D)$ in $Y'$ is inside the locus
  where $f'$ is an isomorphism.
  The exceptional locus $E'$ is disjoint from the preimage of $x$ in $Y'$,
  so the pair $(Y',B'+E')$ is also snc there.
  We chose the resolution $h$ of $(Y',B'+E')$ to be an isomorphism
  over $\snc(Y',B'+E')$, so the composition $g=f\circ f'\circ h$
  is an isomorphism over $x$.
  
  In other words, $g$ satisfies Condition 1 of thriftiness from (\ref{def-thrifty}).
  Since $g$ is a log resolution of $(X,D)$,
  it follows from (\ref{condition2}) that $g$ is thrifty.
  So there is a thrifty log resolution of $(X,D)$ factoring through $f$.
\end{proof}

\begin{corollary}
  Every thrifty resolution is dominated by a thrifty log resolution.
\end{corollary}

\begin{proposition}\label{rational-log-thrifty}
  If a log resolution of a pair is rational, then it is thrifty.
\end{proposition}

\begin{proof}
  Let $f\colon (Y,B)\to (X,D)$ be a rational log resolution.
  We'll verify that $f$ satisfies Condition 1 and Condition 2.
  Since $f$ is log, Condition 2 is automatic by (\ref{condition2}).
  Condition 1, on the other hand, is a property of $\snc(X,D)$.
  Rational resolutions are defined in terms of sheaves,
  and $U=\snc(X,D)$ is an open set, so the restriction
  of $f$ to $f^{-1}(U)$ is still a rational resolution of $(U,D\cap U)$.
  
  Every snc pair is dlt, so by \cite[2.87]{KK13},
  which says that a resolution of a dlt pair is rational
  if and only if it is thrifty,
  we conclude that $f$ is thrifty over $\snc(X,D)$,
  and hence satisfies Condition 1 there.
  But it is sufficient to check the snc locus of $(X,D)$ to verify Condition 1,
  so $f$ is thrifty.
\end{proof}

Next we'll show that thrifty resolutions 
satisfy an analogue of the Grauert-Riemenschneider
vanishing theorem (see \cite{GR}).
Thus the second condition in (\ref{ratresofpairs})
is automatically true, at least
for resolutions that are known to be thrifty.

To prove this, we'll start with a recent result from the literature.
The assumption that our varieties are defined over a field
of characteristic $0$ is necessary here.

\begin{proposition}[Special case of \protect{\cite[10.34]{KK13}}]
  Let $(Y,B)$ be snc, and $f\colon Y\to X$ a projective morphism
  of varieties over a field of characteristic $0$.
  For any lc center $Z$ of $(Y,B)$,
  write $F_Z\subset Z$ for the generic fiber of $f|_Z\colon Z\to f(Z)$.
  Set
  
  \begin{align*}
    c=\max\{\dim F_Z\colon Z\text{ is an lc center}\}.
  \end{align*}
  
  \ 
  
  Then $R^if_\ast\omega_Y(B)=0$ for $i>c$.
  \label{GR-thrifty-prelim}
\end{proposition}

The statement of \cite[10.34]{KK13} is more general,
but we only need this version here.

\begin{proposition}[GR-type vanishing for thrifty resolutions]
  \label{GR-thrifty}
  If $f\colon (Y, B)\to (X,D)$ is a thrifty resolution,
  then $R^if_\ast\omega_Y(B)=0$ for all $i>0$.
\end{proposition}

\begin{proof} 
  We'll appeal to (\ref{GR-thrifty-prelim}),
  since $(Y,B)$ is an snc pair and the resolution 
  $f\colon Y\to X$ is projective.
  Let $Z$ be a stratum of $(Y,B)$, and let $F_Z$ be the generic
  fiber of the map $f|_Z\colon Z\to f(Z)$.
  The image $f(Z)$ is closed, since $f$ is projective, and it is
  also irreducible because $Z$ is, so it has a generic point.
  
  By (\ref{GR-thrifty-prelim}), $R^if_\ast\omega_Y(B)=0$
  for all $i>c$, where
  
  \begin{align*}
    c=\max\{\dim F_Z \colon  Z\text{ is a stratum}\}.
  \end{align*}
  
  \ 
  
  The lc centers of an snc pair are exactly its strata: see \cite[4.15]{KK13}.
  Since $f$ is thrifty, it is birational on every stratum of $(Y,B)$.
  In particular, it is dominant when restricted to each $Z$,
  so the dimension of each generic fiber is 0: see \cite[p.\ 290]{Eisenbud}.
  Thus $c=0$, so $R^if_\ast\omega_Y(B)=0$ for all $i>0$.
\end{proof}

\ 

\section{Hyperplane sections of rational pairs}

In this section we prove that a general hyperplane section of a rational pair
is again a rational pair.
First we'll verify that
thrifty resolutions restrict well to hyperplane sections of pairs in $\bP^N$.

\begin{notation}\label{hyperplane-notation} 
Let $X$ be a projective variety, let
$(X, D)$ be a reduced pair,
and let $f\colon (Y, C)\to (X, D)$ be a resolution.
If $X_H=X\cap H$ is a general hyperplane section of $X$ in $\bP^N$, then
let $D_H=D\cap H$; by generality of $H$, $D_H$ is a divisor in $X_H$.
Also, let $Y_H=f^{-1}_\ast X_H=f^\ast X_H$---it 
is a pullback, by generality
of $H$---and let $f_H=f|_{Y_H}$.
Finally, let $C_H=(f_H)^{-1}_\ast D_H$.
\end{notation}

\begin{lemma}\label{hyperplane-res}
  Using the notation of (\ref{hyperplane-notation}), 
  $f_H\colon (Y_H, C_H)\to (X_H, D_H)$ is
  a resolution.
\end{lemma}

\begin{proof}
  We need to verify that $Y_H$ is smooth and $(Y_H, C_H)$ is an snc pair.
  Smoothness of $Y_H$ comes from the fact that a resolution of a 
  projective variety
  restricts to a resolution of a general one of its hyperplane sections.
  The same is true for each component of $C_H$.
  As for their intersections, since $H$ is general
  they are transverse.
\end{proof}

\begin{lemma}\label{hyperplane-thrifty}
  With the notation of (\ref{hyperplane-notation}),
  if $f$ is a thrifty log resolution,
  then $f_H$ is also thrifty.
\end{lemma}

\begin{proof}
  Since $f$ is log, $\Ex(f)+C$ is an snc divisor in $Y$.
  Also, for general $H$, by the same reasoning as above
  we have that $C+\Ex(f)+Y_H$ is snc in $Y$.
  Thus the exceptional components of $f_H$ have the right 
  codimension---that is, 1---and
  $\Ex(f_H)$ is a divisor in $Y_H$.
  Also, $\Ex(f_H) + C_H$ is snc.
  
  So $f_H$ is also a log resolution, which means that Condition 2
  for thriftiness is automatic.
  We just need to verify Condition 1: no stratum of $(X_H, D_H)$ is
  in $f_H(\Ex(f_H))$.
  But for almost all $H$ this is true: the only $H$ for which
  this fails are those such that
  a component of $\Ex(f)$ maps into, but not onto, a stratum of
  $(X, D)$, and then $H$ cuts out exactly that image
  when it is intersected with $D$.
  But by generality this does not happen.
  
  So $f_H$ is thrifty as well.
\end{proof}

It is well known that rational singularities satisfy a Bertini-like theorem:
if $X\subset\bP^N$ is a projective variety with rational singularities,
then a general hyperplane section of $X$ also has rational singularities.
We'll show next that the analogous statement holds for rational pairs.

\begin{theorem}[Bertini-type result for rational pairs]
  \label{bertini-type}
    If $(X,D)$ is a projective rational pair,
  then a general hyperplane section of $(X,D)$ is also rational.
\end{theorem}

\begin{proof}
  We'll continue to use the notation of (\ref{hyperplane-notation}).
  For a general hyperplane $H\subset\bP^N$,
  $X_H$ is normal and $D_H\subset X_H$
  is a reduced divisor, so $(X_H,D_H)$ is also a reduced pair.
  
  Let $f\colon (Y, C)\to (X,D)$ be a thrifty log resolution, which
  is also rational because $(X,D)$ is rational.
  Then $f_H\colon (Y_H,C_H)\to (X_H, D_H)$ is also a thrifty resolution,
  by (\ref{hyperplane-res}) and (\ref{hyperplane-thrifty}).
  We'll show that $f_H$ is rational, and hence that $(X_H,D_H)$
  is a rational pair.
  
  Start with the short exact sequence
  \begin{align*}
    \xymatrix{0\ar[r]&\sO_Y(-Y_H)\ar[r]&\sO_Y\ar[r]&\sO_{Y_H}\ar[r]&0}
  \end{align*}
  Twist by $-C$, which is Cartier (by smoothness of $Y$), and get
  a new short exact sequence.
  Then push forward to $X$, obtaining the long exact sequence
  \begin{align*}
    \xymatrix{0\ar[r]& f_\ast\sO_Y(-Y_H-C)\ar[r]& f_\ast\sO_Y(-C)\ar[r]&
    f_\ast\sO_{Y_H}(-C_H)\ar[r]&\\
    \ar[r]& R^1f_\ast\sO_Y(-Y_H-C)\ar[r]&\cdots &&}
  \end{align*}
  Now $Y_H$ is a pullback.
  By the projection formula, together with the rationality
  assumption on $(X,D)$,
  \begin{align*}
    R^if_\ast\sO_Y(-Y_H-C)&\simeq R^if_\ast\sO_Y(-C)\otimes\sO_X(-X_H)\\
    &\simeq \begin{cases}
      \sO_X(-D-X_H) & i=0\\
      0 & \text{else}
    \end{cases}
  \end{align*}
  So $R^if_\ast\sO_{Y_H}(-C_H)=0$ for $i>0$ by the long exact sequence.
  Also, there is now have a short exact sequence of pushforwards by $f$,
  yielding isomorphisms
  \begin{align*}
    f_\ast\sO_{Y_H}(-C_H) &\simeq \sO_X(-D)/\sO_X(-D-X_H)\\
    &\simeq\sO_X(-D)\otimes\sO_{X_H}\\
    &\simeq\sO_{X_H}(-D_H).
  \end{align*}
  So $f_H$ is a thrifty rational resolution, and therefore $(X_H, D_H)$
  is a rational pair.
\end{proof}

\

\section{Preliminary results: base change by the local scheme near a point}

We are now ready to develop the main theorem:
that rationality of pairs is preserved by deforming a flat family
defined over a base scheme $S$.
First we'll reduce to an especially simple case,
where $S$ is the spectrum of a regular
local ring,
and then we'll prove the claim
in that situation with an induction argument.
In order to reduce to the case where the base
$S$ is $\Spec R$ for a regular local ring $R$, we need several preliminary
results about how rational pairs and thrifty resolutions behave
with respect to the base change by a morphism $\Spec\sO_{S,s}\to S$,
where $s\in S$ is a smooth point.

The results that follow are all closely related,
and they share a common notation.
For convenience, we collect all the notation here
and will refer back to it throughout the rest of the paper.

\begin{notation}
  Let $f\colon X\to S$ be a morphism, and $s\in S$ a point.
  Let $(X, D)$ be a reduced pair such that the fibers $(X_s, D_s)$ form a reduced pair.
  Also, let $f \colon (Y, B)\to(X, D)$ be a resolution.
  Let
  $\Spec\sO_{S,s}\to S$
  be the inclusion of the local scheme near $s$,
  and base change by this morphism:
  let $X'=X\times_S\Spec\sO_{S,s}$, and similarly for $D', Y', B'$.
  Write $\pi\colon X'\to X$ for the natural projection onto the first factor.
  \label{bc-notation}
\end{notation}

As we will see, all the salient aspects of this situation
are preserved by the base change by $\Spec\sO_{S,s}\to S$.
The next few results are standard commutative algebra,
but for lack of a reference we include the proofs here.

\begin{lemma} \label{base-change-stalks}
  With the notation of (\ref{bc-notation}), let $x'\in X'$.
  If $\pi(x')=x$, and $\sF$ is a sheaf of $\sO_X$-modules on $X$,
  then $\sF_x\simeq (\pi^\ast\sF)_{x'}$.
  That is, stalks of $\sO_X$-modules are preserved by the base change.
  In particular, $\sO_{X,x}\simeq\sO_{X',x'}$.
\end{lemma}

\begin{proof}
  We may assume the schemes are affine, so the base change corresponds to
  a pushout diagram of rings:
  
  \begin{align*}
    \xymatrix{B\ar[r]& B_\fp\\
      A\ar[u]\ar[r] & A_\fp\ar[u]}
  \end{align*}
  
  \ 
  
  Here $A$ and $B$ are rings, $\phi \colon A\to B$ is the homomorphism corresponding to the affine
  morphism $\Spec B\to\Spec A$, and $\fp\lhd A$ is a prime ideal.
  Then the $\sO_X$-module $\tilde{F}$ has the form $\tilde{M}$ for a $B$-module $M$,
  and $\pi^\ast\sF=(M\otimes_BB_\fp)^\sim\simeq(M\otimes_AA_\fp)^\sim$.
  Let $\fq'$ be a prime ideal in the ring $B_\fp$.
  It corresponds to an ideal $\fq$ in $B$ that is disjoint from $f (A-\fp)$.
  Here $\fq'$ corresponds to $x'\in X'$; $\fq$ to $x\in X$.
  It suffices to verify that $(M\otimes_AA_\fp)_{\fq'}\simeq M_\fq$.
  By basic properties of localization (see \cite[Tag 02C7]{Stacks})
  and the fact that $\phi(A-\fp)\subset B-\fq$,
  this is true.

  Note in particular the case where $\sF\simeq\sO_X$:
  if $\pi(x')=x$, then $\sO_{X,x}\simeq\sO_{X',x'}$.
\end{proof}

\begin{corollary}\label{stalks}
  With the notation of (\ref{bc-notation}),
  let $\sP$ be a property of $X$ that may be
  checked locally (i.e., at stalks).
  Then $\sP$ holds on the image of the projection $\pi(X')\subset X$ 
  if and only if $\sP$ holds on $X'$.
\end{corollary}

\begin{example}\label{stalk-applications}
  We continue to use the notation of (\ref{bc-notation}).
  Before moving on, we'll note a
  few specific applications of (\ref{stalks})
  to the varieties $X$ and $Y$, and 
  their counterparts $X'$ and $Y'$ under the base change.  
  Here are a few useful choices for the property $\sP$ in (\ref{stalks}).
  We'll refer back to these in the rest of the paper.
  
\begin{enumerate}
  \item {\bf Nonsingularity.}
    From (\ref{stalks}) we immediately 
    see that $Y'$ is nonsingular (because $Y$ is),
    and each component of $B'\subset Y'$ is nonsingular
    (because each component of $B$ is).

    \ 

  \item {\bf Codimension.}
    if $x'\in X'$ maps to $x\in X$, then $x'$ and $x$ have the same
    codimension.
    This is immediately clear from (\ref{base-change-stalks}),
    because $\codim(x,X)=\dim\sO_{X,x}$ and similarly for $\codim(x',X')$.

    \ 
    
  \item {\bf Reduced divisors.}
    Moreover, since we are assuming that $D\subset X$ is a reduced divisor,
    the base change $D'$ is also reduced.
    The coefficient of a component of a divisor
    is checked at the stalk at the generic point of each component,
    and we have just seen that
    codimension-1 points in $X'$ map to codimension-1 points in $X$.
    Note that the function field also is involved
    in checking the coefficient of a divisor.
    The function field is also preserved---it is the stalk of the structure sheaf
    of any generic point of the base change---so reducedness is preserved.

\ 

  \item {\bf The snc locus.}
    The property of being in the snc locus of a pair is checked using
    stalks of the structure sheaf and of quotients of the structure sheaf.
    Nonsingularity of the variety and of each component of the divisor
    is checked on stalks of the structure sheaf.
    The other requirement for a point to be in the snc locus of a pair
    is that any components of the divisor
    that pass through the point must meet transversally.
    That is, the local equations
    that cut out the components form part
    of a regular sequence in the stalk of the structure sheaf
    at each point of the variety (see \cite[3.1.5]{SGA}).
    Now the stalks of the structure sheaf are preserved
    by the base change,
    and a component of $D'$ passes through $x'$ in $X'$
    if and only if its image passes through $x=\pi(x')$ in $X$.
    So $x'$ is in $\snc(X',D')$ if and only if $x=\pi(x')$ is
    in $\snc(X,D)$.
    Since $(Y,B)$ is assumed to be snc in (\ref{bc-notation}),
    we also have that $(Y',B')$ is snc.
\end{enumerate}
\end{example}

\ 

\begin{corollary}\label{base-change-nbhd}
  Using (\ref{bc-notation}),
  let $\sF\to\sG$ be a morphism of coherent sheaves of $\sO_X$-modules on $X$,
  and assume that the induced morphism $\pi^\ast\sF\to\pi^\ast\sG$ on $X'$ is an isomorphism.
  Then $\sF\to\sG$ is an isomorphism in a neighborhood of the fiber $X_s$.
\end{corollary}

\begin{proof}
  First, note that the natural inclusion morphism $X_s\to X$ factors 
  through the projection
  $\pi\colon X'\to X$.
  This is just because the local ring $\sO_{S,s}$ maps to the residue
  field $\sO_{S,s}/\fm_s = k(s)$, so the fiber
  $X_s=X\times_S\Spec k(s)$ maps to $X'=X\times_S\Spec\sO_{S,s}$.
  So the fiber $X_s$ is contained in the image of $\pi$.
  By (\ref{base-change-stalks}), for any $x'\in X'$ with $\pi(x')=x$, there are isomorphisms
  
  \begin{align*}
    \sF_x\simeq(\pi^\ast\sF)_{x'},\qquad \sG_x\simeq(\pi^\ast\sG)_{x'}.
  \end{align*}
  
  \ 
  
  So $\sF_x\to\sG_x$ is an isomorphism for all $x\in X_s$.
  Now the stalk of the kernel sheaf is the kernel of the stalk morphism,
  and there is a natural isomorphism between the
  stalk of the cokernel sheaf and the cokernel of the morphism on stalks
  (\cite[2.5.A--B]{Vakil}).
  So $\ker(\sF\to\sG)_x=0$ and $\coker(\sF\to\sG)_x=0$ for all $x\in X_s$.
  
  Now $\ker(\sF\to\sG)$ and $\coker(\sF\to\sG)$ are coherent because $\sF$ and $\sG$ are.
  For any coherent sheaf $\sH$, if $\sH_x=0$, then $\sH|_U=0$ for some neighborhood $U$ of $x$.
  So the sheaves $\ker(\sF\to\sG)$ and $\coker(\sF\to\sG)$ are zero in a neighborhood
  of $x$.
  Since this is true for all $x\in X_s$, it follows that $\sF\to\sG$ is an isomorphism
  in a neighborhood of the fiber $X_s$.
\end{proof}

\begin{lemma}
  \label{base-change-thrifty-log}
  With the notation of (\ref{bc-notation}),
  the morphism $f '\colon Y'\to X'$ is a resolution of singularities.
  Moreover, $B'$ is the birational transform of $D'$ and
  $(Y',B')$ is snc; thus, $(Y',B')\to(X',D')$ is a resolution of pairs.
  If $f$ is a log resolution, then so is $f'$,
  and if $f$ is thrifty, so is $f'$.
\end{lemma}

\begin{proof}
  First of all, $f'$ is proper and $Y'$ is smooth:
  properness is always preserved by base change, and $Y'$
  is smooth by (\ref{stalk-applications}).

  We'll now verify that $f'$ is birational.  
  Let $\eta'$ be some codimension-0 point in $X'$.
  Irreducibility
  is not necessarily preserved by the base change, so there might be multiple
  such points; we may choose any one.
  This $\eta'$ maps to some point $\eta$ in $X$,
  which also has codimension 0 by
  (\ref{stalk-applications}).
  Now this $\eta$ is the generic point of $X$, because $X$ is irreducible. 
  Because $Y\to X$ is surjective,
  this point $\eta$ has a preimage $\xi$ in $Y$, 
  and this $\xi$ is the generic point of $Y$.
  Then there is a point in $Y'\simeq Y\times_XX'$ mapping to $\xi$;
  call this point $\xi'$.
  By (\ref{stalk-applications}), $\xi'$ has codimension 0,
  and the map on stalks
  $\sO_{X',\eta'}\to\sO_{Y',\xi'}$
  is exactly the same map as
  $\sO_{X,\eta}\to\sO_{Y,\xi}$,
  which is an isomorphism because $Y\to X$ is birational.
  Then $f'$ is an isomorphism over every generic point of $X'$.
  So it is birational over each component of $X'$.
  
  Now $(Y',B')$ is snc, by (\ref{stalk-applications}).
  The same argument as above shows that $B'\to D'$ is birational over every component
  of $D'$, so $(Y',B')\to (X',D')$ is a resolution of pairs.
    
  We'll show next that the exceptional locus of $f$
  in $Y$ base changes to the exceptional locus of $f'$ in $Y'$,
  and the image of the exceptional locus of $f$ in $X$
  base changes to the image of the exceptional locus of $f'$ in $X'$.
  
  First, taking the base change of all the exceptional components in $Y$
  gives us exactly the exceptional components in $Y'$.
  If a point on $X$ is in the image of $\Ex(f)$
  and also of the projection $X'\to X$,
  then its preimages in $X'$ are in the image of $\Ex(f')$.
  Indeed, given an exceptional component of the resolution $Y\to X$,
  the map $\sO_{X',f'(y')}\to\sO_{Y',y'}$ at its generic point
  $y'$ is not an isomorphism,
  and this persists in the base change to the resolution $Y'\to X'$
  by (\ref{base-change-stalks}):
  it is precisely the same map as $\sO_{X,f(y)}\to\sO_{Y,y}$,
  if $\pi(y')=y$.
  From this the converse follows too: 
  an exceptional component in $Y'$
  maps to an exceptional component in $Y$,
  and a point in $f'(\Ex(f'))\subset X'$ maps to a point in $f(\Ex(f))\subset X$.
  So the base change of the exceptional locus is the exceptional locus,
  and the base change of the image of the exceptional locus
  is the image of the exceptional locus.
  
  Now suppose $f$ is a log resolution.
  We'll show the same is true of $f'$.
  We already know that the base change of a divisor is a divisor, 
  because codimension of points is preserved.
  Consider an irreducible component $F$ of the exceptional locus in $Y'$.
  We'll show $F$ has codimension 1 in $Y'$.
  Assume the opposite: that $F$ has codimension $k\geq 2$ in $Y'$.
  If $\nu$
  is the generic point of $F$, then the map on stalks
  \begin{align*}
    \sO_{X',f '(\nu)}\to\sO_{Y',\nu}
  \end{align*}
  is \emph{not} an isomorphism, because $Z$ is exceptional,
  but there is a closed, codimension-1 subset $V$ of $Y'$ containing $F$
  so that the map on stalks at the generic point of $V$ is an isomorphism.
  Then the map on stalks at the corresponding codimension-$k$ point in $Y$
  is not an isomorphism, but at the codimension-1 point it is.
  In other words, the resolution $Y\to X$ also has an
  exceptional component that is not a divisor---but this is impossible,
  because we assume $Y\to X$ is a log resolution,
  so its exceptional components must be divisors.
  
  So the exceptional locus is a divisor in $Y'$.
  Taking the snc locus commutes with this base change
  by (\ref{stalk-applications}), so $f'$
  is a log resolution.
  
  Finally we show that if $f$ is thrifty, then so is $f'$.
  To verify Conditions 1 and 2 of (\ref{def-thrifty}),
  we must examine the resolution $f'$ at generic points of strata
  of $\snc(X',D')$ and of $(Y',B')$.
  
  We'll start with Condition 1.
  Let $Z$ be a stratum of $\snc(X',D')$,
  so that $Z$ is an irreducible component of
  some intersection $D'_{i_1}\cap\cdots\cap D'_{i_r}$.
  Then, with $\pi\colon X'\to X$ the projection,
  $\pi(Z)$ is irreducible and
  contained in a component of 
  the intersection $D_{i_1}\cap\cdots\cap D_{i_r}$ in $X$.
  Now, by applying (\ref{stalk-applications}),
  we see that
  $\pi(Z)$ has the same codimension in $X$
  as $Z$ does in $X'$, and $\pi(Z)\subset\snc(X,D)$.
  By \cite[4.16.2]{KK13}, which says that the intersection of
  $r$ components of a dlt divisor
  has pure codimension $r$,
  $\pi(Z)$ is actually equal to a stratum of $\snc(X,D)$ 
  (and is not merely contained in one).
  Now the assumption that $f$ is thrifty means that
  $\pi(Z)$ is not contained in $f(\Ex(f))$.
  By the argument above, which showed that taking the image of
  the exceptional locus commutes with the base change
  $X'\to X$,
  $Z$ is not in $f'(\Ex(f'))$, and so $f'$ satisfies Condition 1.
  
  Condition 2 is almost the same.
  If a stratum of $(Y',B')$ lies in $\Ex(f')$,
  then its image is a stratum of $(Y,B)$
  and lies in $\Ex(f)$,
  but this is impossible because $f$ is thrifty.
  So $f'$ also satisfies Condition 2, and hence is thrifty.
\end{proof}

\ 

\section{Deformation invariance of rational pairs}

In this section we prove the main theorem:
given a pair $(X,D)$ with $D$ Cartier and a flat morphism $X\to S$,
if the fibers $(X_s,D_s)$ over a smooth point $s\in S$ are a rational (reduced) pair,
then $(X,D)$ is also rational near $X_s$.

Because we assume $D$ is Cartier, we may use (\ref{kempf2}), the analogue of Kempf's
criterion for rational pairs,
to conclude that $(X,D)$ is rational near $X_s$.
In order to check the second part of (\ref{kempf2}), we will need to
exhibit a thrifty resolution $f$ of $X$ so that
that $f_\ast\omega_Y(B)\simeq\omega_X(D)$, at least near $X_s$.
The next lemma shows that actually we only need to verify
that $f_\ast\omega_Y(B)\to\omega_X(D)$ is surjective:
injectivity is automatic.

\begin{lemma}
  \label{can-sheaves-injective}
  Suppose $f\colon Y\to X$ is a birational and proper morphism between normal varieties, 
  $D$ is a divisor in $X$,
  $B=f^{-1}_\ast D$ is the birational transform of $D$ in $Y$,
  and $\omega_Y$ is torsion free.
  (For example, let $f\colon (Y,B)\to (X,D)$ be a resolution, so that $\omega_Y$ is invertible.)
  Then there is a logarithmic trace morphism $f_\ast\omega_Y(B)\to\omega_X(D)$,
  and it is injective.
\end{lemma}

\begin{proof}
  Let $U\subset X$ be the largest open set 
  over which $f$ is an isomorphism.
  The complement of $U$ has codimension at least $2$,
  because $X$ is normal.
  Let $i\colon U\hookrightarrow X$ be the inclusion.
  On $U$, we have
  
  \begin{align*}
    f_\ast\omega_Y(B)|_U\simeq \omega_X(D)|_U,
  \end{align*}
  
  \ 
  
  because the restricted map $f\colon f^{-1}(U)\to U$ is an isomorphism.
  Now $i^\ast$ and $i_\ast$ are adjoint functors, so there is a natural morphism
  
  \begin{align*}
    f_\ast\omega_Y(B)\to i_\ast i^\ast f_\ast\omega_Y(B),
  \end{align*}
  
  \ 
  
  and the sheaf on the right can also be written as $i_\ast(f_\ast\omega_Y(B)|_U)$.
  Putting these maps together, we obtain a composition
  
  \begin{align*}
    f_\ast\omega_Y(B)\to i_\ast(f_\ast\omega_Y(B)|_U)\to i_\ast(\omega_X(D)|_U).
  \end{align*}
  
  \ 
  
  On the open set $U$, $\omega_X(D)$ and $i_\ast(\omega_X(D)|_U)$ are equal.
  The complement $X\setminus U$ has codimension at least 2
  and the sheaves are reflexive, so they are equal on $X$ by \cite[1.6]{ReflexiveSheaves}.
  From this we have the desired map $f_\ast\omega_Y(B)\to\omega_X(D)$.
  
  Now we use the assumption that $\omega_Y$ is torsion free.
  This guarantees that $\omega_Y(B)$ is also torsion free, 
  as is its pushforward $f_\ast\omega_Y(B)$.
  Indeed, for any open set $V\subset X$, the sections of $f_\ast\omega_Y(B)$
  on $V$ are by definition the same as those of $\omega_Y(B)$ on $f^{-1}(V)$.
  Now $f_\ast\omega_Y(B)\to\omega_X(D)$ is an isomorphism at the generic
  point of $X$,
  so the kernel of the morphism is a torsion sheaf.
  But $f_\ast\omega_Y(B)$ is torsion free, so the logarithmic trace map is injective.
\end{proof}

\begin{theorem}[Deformation invariance for rational pairs]
  \label{main-thm}
  Let $(X, D)$ be a pair, with $D$ Cartier.
  Suppose $X\to S$ is a flat morphism, and $s\in S$ is a smooth point
  so that the fibers $(X_s, D_s)$ form a reduced pair.
  If $(X_s, D_s)$ is a rational pair, then $(X,D)$ is a rational pair
  in a neighborhood of $(X_s, D_s)$.
  
  That is, if $(X_s,D_s)$ is rational at $x$, then
  $(X,D)$ is also rational at $x$.
\end{theorem}

\begin{proof}
  The first step is to show that we may assume the base $S$
  is the spectrum of a regular local ring $R$.
  To see this, first
  base change the morphism $X\to S$ by the flat morphism
  $\Spec\sO_{S,s}\to S$, and let $X'=X\times_S\Spec\sO_{S,s}$
  as in the notation of (\ref{bc-notation}).
  Similarly, let $D'=D\times_S\Spec\sO_{S,s}$.
  Then $X'\to\Spec\sO_{S,s}$ is again flat.
  
  We'll show that it suffices to prove the result for the new pair $(X',D')$:
  if $(X',D')$ is rational near $X_s$, then $(X,D)$ is
  also ratonal near $X_s$.
  Let $f\colon (Y,B)\to (X, D)$ be a thrifty log resolution,
  not necessarily rational.
  Then there is a Cartesian diagram:
  
  \begin{align*}
    \xymatrix{Y\ar[d]_f&Y'\ar[d]^{f'}\ar[l]_{\rho}\\
    X\ar[d] & X'\ar[d]\ar[l]_{\pi}\\
    S &\Spec\sO_{S,s}\ar[l]}
  \end{align*}
  
  \ 
  
  By (\ref{stalks}), if $X'$ is normal and CM, then so is $X$
  at every point in the image of the projection $X'\to X$,
  and if $D'$ is reduced, 
  then so is $D$ at every point in the image of $D'\to D$.
  Also, by (\ref{base-change-thrifty-log}),
  $f'\colon (Y',B')\to (X',D')$ is a thrifty log resolution.
  
  Suppose for now that we have shown $(X', D')$ is a rational pair
  in a neighborhood $U$ of $X_s$, so that every
  thrifty resolution of the pair is rational over $U$.
  Then $f'$ is a rational resolution; that is, on $U$ we have
  
  \begin{align*}
    \dR(f')_\ast\sO_{Y'}(-B')\simeq \sO_{X'}(-D').
  \end{align*}
  
  \ 
  
  The map $\Spec\sO_{S,s}\to S$ is flat, so by cohomology and base change 
  for flat morphisms (\cite[III.9.3]{Hart}),
  there is an isomorphism on $U$:
  
  \begin{align*}
    \pi^\ast \dR f_\ast\sO_Y(-B)\simeq\sO_{X'}(-D').
  \end{align*}
  
  \ 
  
  To prove the original thrifty resolution $f$ is then rational in a neighborhood of $X_s$,
  we need to verify that $\dR f_\ast\sO_Y(-B)\simeq\sO_X(-D)$ near $X_s$.
  Now $\sO_{X'}(-D')=\pi^\ast\sO_X(-D)$, and the pushforwards $R^if_\ast\sO_Y(-B)$ are coherent.
  By (\ref{base-change-nbhd}) it follows that $\dR f_\ast\sO_Y(-B)\simeq\sO_X(-D)$ in a neighborhood
  of $X_s$.
  
  So it suffices to prove the result in the case where the base is
  $\Spec\sO_{S,s}$.
  We may then assume that $S=\Spec R$, where $R$ is a regular local
  ring of dimension $n$.
  
  We'll prove that $X$ is normal in a neighborhood of the fiber $X_s$.
  Then, by (\ref{kempf2}) and (\ref{GR-thrifty}), we will 
  only need to prove that 
  $X$ is CM near $X_s$ and
  for some thrifty resolution $f\colon (Y,B)\to(X,D)$, the logarithmic trace
  $f_\ast\omega_Y(B)\to\omega_X(D)$
  is an isomorphism in a neighborhood of $X_s$.
  By (\ref{can-sheaves-injective}) the logarithmic trace is injective,
  so we will show that $X$ is normal and CM and the map of sheaves is surjective
  in a neighborhood of $X_s$.

  Following the idea of the proof of 
  \cite[Th\'eor\`eme 2]{Elkik}, we'll prove this using induction on $n$.  
  Our base case is $n=0$.
  If $n=0$, then $S=\Spec K=\{s\}$ for some field $K$.
  In this case $(X, D)=(X_s, D_s)$, so the conclusion is trivially true.

  Now let $\dim R=n$.
  For the inductive hypothesis, assume the result is true for $\dim R < n$:
  if the base scheme is the spectrum of a regular local ring of dimension less than $n$,
  and if $(X_s,D_s)$ is rational,
  then $(X,D)$ is rational in a neighborhood of $X_s$.

  Let $t$ be a regular parameter in $R$.
  Then let $X_t=X\times_S \Spec (R/tR)$, so that $X_t$ is the pullback of the
  divisor defined by $t$ in $S=\Spec R$, and define $D_t$ similarly.
  Then $X_t$ is Cartier in $X$, and $D_t$ in $D$.
  Note that $(X_s, D_s)$ is a subpair of $(X_t, D_t)$: 
  $X_s\subset X_t$ and $D_s\subset D_t$.
  
  Now the regular local ring $R/tR$ has dimension $n-1$, so
  $(X_t,D_t)$ is rational near $X_s$ by the inductive hypothesis.
  In particular, $X_t$ is normal and CM.
  Then $X$ is also normal and CM in a neighborhood of $X_t$: 
  these are properties that
  pass from a Cartier divisor to an open set in the whole space.
  Any neighborhood of $X_t$ is also a neighborhood of $X_s$,
  so $X$ is CM and normal near $X_s$.
  Working in this neighborhood,
  normality allows us to use (\ref{kempf2}),
  and since we also have in this neighborhood that $X$ is CM, 
  $f$ is thrifty, and $f_\ast\omega_Y(B)\to\omega_X(D)$
  is injective, it just remains to show that
   $f_\ast\omega_Y(B)\to\omega_X(D)$
  is surjective near $X_s$.
  
  Let $(Y,B)\to (X,D)$ be a thrifty log resolution.
  There is then a proper morphism $Y_t\to X_t$
  and a component $Y_1$ of $Y_t$ mapping birationally to $X_t$.
  Write $f_1=f|_{Y_1}\colon Y_1\to X_t$,
  and let $B_1=(f_1)^{-1}_\ast(D_t)$ be the birational transform of $D_t$.
  Now $(Y_1,B_1)$ is not snc---$Y_1$ need not even be smooth---but the morphism
  $(Y_1,B_1)\to (X_t,D_t)$ is birational and satisfies
  Condition 1 from (\ref{def-thrifty}).
  By (\ref{logres}), there is a thrifty log resolution
  $f_2\colon (Y_2,B_2)\to (X_t,D_t)$ factoring through $(Y_1,B_1)$.
  Write $\tilde{f}$ for the intermediate birational morphism $Y_2\to Y_1$.
  
  \begin{align*}
    \xymatrix{(Y_2,B_2)\ar[r]^-{\tilde{f}}\ar[rd]_-{f_2} & (Y_1,B_1)\ar[d]^-{f_1}\\
    & (X_t,D_t)}
  \end{align*}
  
  \

  By assumption, $(X_t, D_t)$ is a rational pair, so $f_2$ is
  a rational resolution.
  Now $D_t$ is Cartier in $X_t$, so by (\ref{kempf2}), the logarithmic trace map
  is an isomorphism:
  
  \begin{align*}
    (f_2)_\ast\omega_{Y_1}(B_1)\simeq\omega_{X_t}(D_t).
  \end{align*}
   
   \ 
   
  Because $t$ is a regular parameter, we have an exact sequence
  
  \begin{align*}
    \xymatrix{0\ar[r]&\omega_Y\ar[r]^{\cdot t}&\omega_Y\ar[r]&\omega_{Y_t}\ar[r]&0}
  \end{align*}
  
  \ 
  
  and similarly for $\omega_X, \omega_{X_t}$.
  Twist the sequences by $B$ and $D$, respectively.
  Both operations are exact because $B,D$ are Cartier.
  Push forward the sequence on $Y$ by $f$.
  The result is a commutative diagram:

  \begin{align*}
    \xymatrix{0\ar[r]&f_\ast\omega_Y(B)\ar[r]^-{\cdot t}\ar[d]&f_\ast\omega_Y(B)\ar[r]\ar[d]
      &f_\ast\omega_{Y_t}(B_t)\ar[r]\ar[d]&0\\
      0\ar[r]&\omega_X(D)\ar[r]_-{\cdot t}&\omega_X(D)\ar[r]&\omega_{X_t}(D_t)\ar[r]&0}
  \end{align*}

  \ 
  
  Both rows are exact: the top is exact by (\ref{GR-thrifty}), 
  the analogue of the Grauert-Riemenschneider
  vanishing theorem for thrifty resolutions.

  By Grothendieck duality, $\omega_{Y_1}(B_1)$ is a subsheaf of $\omega_{Y_t}(B_t)$.
  Moreover, by (\ref{can-sheaves-injective}),
  the logarithmic trace map $\tilde{f}_\ast\omega_{Y_2}(B_2)\to\omega_{Y_1}(B_1)$ is injective.
  Composing these injective maps, we get
  an injection  $\tilde{f}_\ast\omega_{Y_2}(B_2)\to\omega_{Y_t}(B_t)$.
  Pushing forward by $f$ and using that $f_2=f\circ \tilde{f}$, we have
  an injective map
  
  \begin{align*}
     f_\ast \tilde{f}_\ast\omega_{Y_2}(B_2)=(f_2)_\ast\omega_{Y_2}(B_2)\hookrightarrow f_\ast\omega_{Y_t}(B_t).
  \end{align*}

  \ 
  
  The isomorphism $(f_2)_\ast\omega_{Y_2}(B_2)\to\omega_{X_t}(D_t)$ 
  then factors through
  $f_\ast\omega_{Y_t}(B_t)$, so the composition
  
  \begin{align*}
     f_\ast\omega_{Y_2}(B_2)\to f_\ast\omega_{Y_t}(B_t)\to\omega_{X_t}(D_t)
  \end{align*}
  
  \ 
  
  is injective and surjective.
  Thus the right vertical arrow in the diagram is surjective.
  
  To prove our desired result---that $f$ is a rational resolution---we
  just need to verify that the middle vertical arrow is also surjective
  in a neighborhood of $X_s$.
  Let $x\in X_s$ be any point, not necessarily closed.
  Then $x$ maps to the single closed point---the maximal ideal $\fm$---of $\Spec R$.
  This is clear from the definition of $R$: it is $\sO_{S,s}$ for a smooth point $s\in S$,
  and $X_s$ is just the fiber over $s$.

  Working in an affine neighborhood of our point $x$, we may assume $X=\Spec A, Y=\Spec C$,
  and $\omega_X(D),\omega_Y(B)$ correspond to finite modules $M,N$ over $A,C$
  respectively.
  We have ring maps $R\to A\to C$.
  Considering $N$ as an $A$-module via the map $A\to C$,
  and then thinking of $N\to M$ as a morphism of $A$-modules, we have
  the local version of the morphism
  $f_\ast\omega_Y(B)\to\omega_X(D)$.
  
  The logarithmic trace map $f_\ast\omega_Y(B)\to\omega_X(D)$ is injective by 
  (\ref{can-sheaves-injective}), so
  $N\to M$ is injective and $N$ is a submodule of $M$.
  The morphism of sheaves $f_\ast\omega_{Y_t}(B_t)\to\omega_{X_t}(D_t)$
  corresponds locally
  to the map of modules $N/tN\to M/tM$.
  Let $\fp$ be the ideal in $A$ corresponding to the point $x$.
  Since $t\in\fm$ and $\fp$ pulls back to $\fm$, $t$ is also in $\fp$ when we view $t$
  as an element of $A$.  
  
  Next we localize the entire diagram of $A$-modules at $\fp$,
  to get a diagram of maps between $A_\fp$-modules:

  \begin{align*}  
    \xymatrix{ & 0\ar[d]&0\ar[d]&K\ar[d]&\\
     0\ar[r]&N_\fp\ar[r]^-{\cdot t}\ar[d]&N_\fp\ar[r]\ar[d]&N_\fp/tN_\fp\ar[r]\ar[d]&0\\
      0\ar[r]&M_\fp\ar[r]^-{\cdot t}\ar[d]&M_\fp\ar[r]\ar[d]&M_\fp/tM_\fp\ar[r]\ar[d]&0\\
      &(M/N)_\fp & (M/N)_\fp & 0 &}
  \end{align*}
  
  \ 
  
  Localization is exact, so the rows of the diagram are short exact, 
  the first two arrows are injective, and the right arrow is surjective.
  Also, localization commutes with taking
  quotients, so $(M/tM)_\fp\simeq M_\fp/tM_\fp$,
  $(N/tN)_\fp\simeq N_\fp/tN_\fp$, and $(M/N)_\fp\simeq M_\fp/N_\fp$.
  By the snake lemma, there is a short exact sequence:
  
  \begin{align*}
    \xymatrix{0\ar[r]&K\ar[r]&(M/N)_\fp\ar[r]^-{\cdot t}&(M/N)_\fp\ar[r]&0}
  \end{align*}
  
  \ 
  
  In particular, notice that the map $(M/N)_\fp\to(M/N)_\fp$ given by multiplication
  by $t$ is surjective, so $(M/N)_\fp = t(M/N)_\fp$.
  
  All the modules we had before localizing were finite over $A$, so all the 
  localized modules
  in the new diagram are finite over $A_\fp$.
  Now $t$ is in the single maximal ideal of $A_\fp$ and
  $(M/N)_\fp = t(M/N)_\fp$, so 
  $(M/N)_\fp=0$ by Nakayama's lemma.
  
  This argument holds for every point $x\in X_s$, so
  the cokernel sheaf $\omega_X(D)/ f_\ast\omega_Y(B)$ is zero at every $x\in X_s$.
  Then, because the cokernel sheaf is coherent, it is actually zero on
  an open set of $X$ containing $X_s$.
  
  Now we've shown that $\omega_X(D)\simeq f_\ast\omega_Y(B)$ in a neighborhood
  of the fiber $X_s$.
  By the reductions above, the induction is complete and
  $(X,D)$ is rational in a neighborhood
  of $X_s$.
\end{proof}

If the family is also proper and $S$ is a curve, then we can say a bit more
about the fibers near $X_s$.

\begin{corollary}\label{main-corollary}
  In the situation of (\ref{main-thm}),
  if the morphism $X\to S$ is also proper
  and $S$ is a curve,
  then there is a neighborhood $W\subset S$ of $s$
  such that for any $w\in W$,
  the pair of fibers $(X_w,D_w)$ is rational.
\end{corollary}

\begin{proof}
  Because $f\colon X\to S$ is now assumed proper, there is
  an open set $W\subset S$ so that for every $w\in W$,
  the fiber $X_w$ is contained in the open locus where $(X,D)$
  is a rational pair.
  By replacing $X$ with $f^{-1}W$ and $S$ with $W$,
  we still have a proper flat family and we may assume $(X,D)$
  is a rational pair.
  
  Now a general hyperplane section of $(X,D)$
  is also rational by the Bertini theorem for rational pairs (\ref{bertini-type}),
  so the fibers $(X_t,D_t)$ over points $t$ near $s$ are also rational.
\end{proof}

\ 

\

\bibliographystyle{alpha}
\bibliography{thesisbib}

\begin{thebibliography}{{Sta}14}

\bibitem[Eis95]{Eisenbud}
David Eisenbud.
\newblock {\em Commutative algebra}, volume 150 of {\em Graduate Texts in
  Mathematics}.
\newblock Springer-Verlag, New York, 1995.
\newblock With a view toward algebraic geometry.

\bibitem[Elk78]{Elkik}
Ren{\'e}e Elkik.
\newblock Singularit\'es rationnelles et d\'eformations.
\newblock {\em Invent. Math.}, 47(2):139--147, 1978.

\bibitem[Fuj07]{Fujino}
Osamu Fujino.
\newblock What is log terminal?
\newblock In {\em Flips for 3-folds and 4-folds}, volume~35 of {\em Oxford
  Lecture Ser. Math. Appl.}, pages 49--62. Oxford Univ. Press, Oxford, 2007.

\bibitem[GR70]{GR}
Hans Grauert and Oswald Riemenschneider.
\newblock Verschwindungss\"atze f\"ur analytische {K}ohomologiegruppen auf
  komplexen {R}\"aumen.
\newblock {\em Invent. Math.}, 11:263--292, 1970.

\bibitem[Har77]{Hart}
Robin Hartshorne.
\newblock {\em Algebraic geometry}.
\newblock Springer-Verlag, New York, 1977.
\newblock Graduate Texts in Mathematics, No. 52.

\bibitem[Har80]{ReflexiveSheaves}
Robin Hartshorne.
\newblock Stable reflexive sheaves.
\newblock {\em Math. Ann.}, 254(2):121--176, 1980.

\bibitem[Kol13]{KK13}
J{\'a}nos Koll{\'a}r.
\newblock {\em Singularities of the minimal model program}, volume 200 of {\em
  Cambridge Tracts in Mathematics}.
\newblock Cambridge University Press, Cambridge, 2013.
\newblock With the collaboration of S{\'a}ndor Kov{\'a}cs.

\bibitem[SGA77]{SGA}
{\em Cohomologie {$l$}-adique et fonctions {$L$}}.
\newblock Lecture Notes in Mathematics, Vol. 589. Springer-Verlag, Berlin,
  1977.
\newblock S{\'e}minaire de G{\'e}ometrie Alg{\'e}brique du Bois-Marie
  1965--1966 (SGA 5), Edit{\'e} par Luc Illusie.

\bibitem[{Sta}14]{Stacks}
The {Stacks Project Authors}.
\newblock \emph{Stacks Project}.
\newblock \url{http://stacks.math.columbia.edu}, 2014.

\bibitem[Sza94]{Szabo}
Endre Szab{\'o}.
\newblock Divisorial log terminal singularities.
\newblock {\em J. Math. Sci. Univ. Tokyo}, 1(3):631--639, 1994.

\bibitem[Vak13]{Vakil}
Ravi Vakil.
\newblock \emph{Foundations of Algebraic Geometry}.
\newblock \url{http://math.stanford.edu/~vakil/216blog/}, June 11, 2013.
\newblock Notes for Math 216 at Stanford University.

\end{thebibliography}

\end{document}